\documentclass[titlepage]{amsart}
\usepackage{amsmath}
\usepackage{amsthm}
\usepackage{amssymb}
\usepackage{setspace}
\usepackage{enumerate}
\usepackage[capitalise]{cleveref}
\usepackage{amsaddr}

\newtheorem{theorem}{Theorem}[section]
\newtheorem{lemma}{Lemma}[section]
\newtheorem{definition}{Definition}[section]

\newtheorem{corollary}{\bf Corollary}[section]

\newcommand{\wsat}{\mathrm{wsat}}
\newcommand{\reg}{\mathrm{reg}}

\newcommand{\pcf}{\mathrm{pcf}}
\newcommand{\tcf}{\mathrm{tcf}}
\newcommand{\cf}{\mathrm{cf}}

\newcommand{\acc}{\mathrm{acc}}
\newcommand{\Ch}{\mathrm{Ch}}
\newcommand\blfootnote[1]{%
  \begingroup
  \renewcommand\thefootnote{}\footnote{#1}%
  \addtocounter{footnote}{-1}%
  \endgroup
}

\title{An Extension of Shelah's Trichotomy Theorem}
\author{Shehzad Ahmed}
\address{Ohio University \\
Dept of Math 321 Morton Hall \\
Athens, Ohio 45701-2979}
\email{sa066513@ohio.edu}

\begin{document}

\begin{abstract}
In \cite{Sh506}, Shelah develops the theory of $\pcf_I(A)$ without the assumption that $|A|<\min (A)$, going so far as to get generators for every $\lambda\in\pcf_I(A)$ under some assumptions on $I$. Our main theorem is that we can also generalize Shelah's trichotomy theorem to the same setting. Using this, we present a different proof of the existence of generators for $\pcf_I(A)$ which is more in line with the modern exposition. Finally, we discuss some obstacles to further generalizing the classical theory.
\end{abstract}
\maketitle

\section{Introduction}
\blfootnote{I would like to thank Todd Eisworth for his assistance with the organization of the manuscript, and an anonymous referee for their helpful remarks which enhanced the clarity of the paper.} 

The pcf theory as presented in \cite{Sh} has proven to be a powerful tool for analyzing the combinatorial structure at singular cardinals as well as their successors. Perhaps the most well-known consequence of the pcf-theoretic machinery is the following theorem due to Shelah:

\begin{theorem}[Shelah]
\begin{equation*}
\aleph_\omega^{\aleph_0}<\max\{\aleph_{\omega_4},(2^{\aleph_0})^+\}.
\end{equation*}
\end{theorem}

This contrasts greatly with the situation for regular cardinals, and tells us that we can get meaningful results about the power of singular cardinals in ZFC. On the other hand, we know that some of this machinery can only work for singular cardinals which are not fixed points of the $\aleph$-function. Given suitable large cardinal hypotheses, one can use Prikry-type forcings to blow up the power of some $\aleph$-fixed points to be arbitrarily large (see \cite{Gi} for an overview). So if the pcf machinery can be generalized to $\aleph$-fixed points, this can only be done in a restricted manner.

In \cite{Sh506}, Shelah does precisely this. The pcf machinery is relativized to particular ideals over some set $A$ which need not satisfy $|A|<\min (A)$. In particular, Shelah is able to obtain generators for every $\lambda\in \pcf_I(A)$. The usual proof of the existence of generators requires obtaining universal cofinal sequences for each $\lambda\in\pcf(A)$, and then showing that exact upper bounds for such sequences yield generators. In the classical case, one can make use of Shelah's trichotomy theorem \cite{Sh}:

\begin{theorem}[Chapter II, Claim 1.2 of \cite{Sh}]\label{thm:Trichotomy} Suppose that $\lambda$ is a regular cardinal with $\lambda>|A|^+$, $I$ is an ideal over $A$, and $\vec f=\langle f_\alpha : \alpha<\lambda\rangle$ is an $<_I$-increasing sequence of functions from $A$ to $\mathrm{ON}$. Then $\vec f$ satisfies at least one of the following conditions:\\
\begin{enumerate}
\item \underline{$\mathrm{Good}$}: $\vec{f}$ has an exact upper bound $f\in{}^A\mathrm{ON} $ such that $\cf(f(a))>|A|$ for all $a\in A$.\\
\item \underline{$\mathrm{Bad}$}: There are sets $S(a)$ for each $a\in A$ such that $|S(a)|\leq|A|$ and an ultrafilter $D$ over $A$ disjoint from $I$ such that, for all $\xi<\lambda$, there exists some $h_\xi\in \prod_{a\in A}S(a)$ and some $\eta<\lambda$ such that $f_\xi<_D h_\xi<_D f_\eta$.\\
\item \underline{Ugly}: There is a function $g: A\to\mathrm{ON}$ such that, letting $t_\xi=\{a\in A : f_\xi(a)>g(a)\}$, the sequence $\vec{t}=\langle t_\xi : \xi<\lambda\rangle$ (which is $\subseteq_I$-increasing) does not stabilize modulo $I$. That is, for every $\xi<\lambda$, there is some $\xi<\eta<\lambda$ such that $t_\eta\setminus t_\xi\notin I$.\\
\end{enumerate}
\end{theorem}

In our desired applications, the functions $\langle f_\xi : \xi<\lambda\rangle$ will belong to $\prod A$, where $A$ is a collection of regular cardinals. So if $\vec f$ does have an exact upper bound $f$, it would be bounded above by the function $a\mapsto a$. This means that if $f$ is Good as above, the requirement that $\cf(f(a)) > |A|$ for each $a\in A$ will force that $|A|<\min(A)$. So this version of trichotomy will not work in the more general setting of \cite{Sh506}. While Shelah pursues a different route, it is natural to ask whether or not one can generalize the trichotomy theorem. Of course, even if we obtain this more general trichotomy theorem, we still have to show that we can find sequences that are neither bad nor ugly. Our main theorem is that one can do precisely that.

This paper is organized as follows: In Section 2, we extend \cref{thm:Trichotomy}, and show that one can still construct sequences that are neither bad nor ugly. In Section 3, we use this to provide a streamlined proof of the fact that generators exist for $\pcf_I(A)$. Finally, we show that the no holes conclusion must fail in general, and that the standard techniques for obtaining transitive generators cannot be generalized.

\section{The Trichotomy Theorem}

Our goal in this section is to generalize \cref{thm:Trichotomy} by replacing the assumption that $|A|^+<\lambda$ with assumptions about the ideal $I$ we are asking about. First, we fix some notation

\begin{definition}
Suppose that $I$ is an ideal over a set $A$ of ordinals. Then

\begin{enumerate}
\item We denote the dual filter by $I^*$.
\item We say that property $P$ holds for $I$-almost every $\alpha\in A$ if the set of $\alpha\in A$ such that $P$ holds is in the dual filter $I^*$.
\item If $f, g$ are functions from $A$ to the ordinals, and $R$ is a relation on the ordinals, then we say $f R_I g$ if and only if $\{\alpha\in A : \neg(f(\alpha) R g(\alpha))\}\in I$.
\item Dually, if $D$ is a filter on $A$, $f, g$ are functions from $A$ to the ordinals, and $R$ is a relation on the ordinals, then we say $f R_D g$ if and only if $\{\alpha\in A : f(\alpha) R g(\alpha)\}\in D$.
\item We say a set $B\subseteq A$ is $I$-positive if $B\notin I$. We denote the collection of $I$-positive sets by $I^+$.
\end{enumerate}
\end{definition}

We now isolate and discuss several properties of ideals that we will be working with.

\begin{definition}
Suppose that $I$ is an ideal on some set $A$. For a cardinal $\theta$, we say that $I$ is weakly $\theta$-saturated if there is no partition of $A$ into $\theta$-many $I$-positive sets.
\end{definition}

Note that if $I$ is weakly $\theta$-saturated, and $\theta_1$ is a cardinal above $\theta$, then $I$ is also weakly $\theta_1$-saturated. Further, note that $I$ is always weakly $|A|^+$-saturated for trivial reasons.

\begin{definition}
If $I$ is an ideal on a set $A$, then let $\wsat(I)$ denote the least cardinal $\theta$ such that $I$ is weakly $\theta$-saturated.
\end{definition}

Another property that we will need indirectly is a weakening of $\theta$-completeness.

\begin{definition}
Suppose that $I$ is an ideal on some set $A$. For a regular cardinal $\theta$, we say that $I$ is $\theta$-indecomposable if $I$ is closed under $\subseteq$-increasing unions of length $\theta$.
\end{definition}

One thing to note above is that, unlike weak saturation, indecomposability is neither upwards nor downwards hereditary. While we will be making use of weak saturation directly in the next section, our use of indecomposability comes by way of combining it with weak saturation. In particular, we will make frequent use of the following result.

\begin{lemma}[Proposition 2.6 of \cite{Eis3}]\label{lem: reg}
Let $I$ be an ideal on a set $A$. The following are equivalent for a regular cardinal $\theta$.

\begin{enumerate}
\item $I$ is weakly $\theta$-saturated and $\theta$-indecomposable.
\item Whenever $\langle B_i : i<\theta\rangle$ is a $\subseteq$-increasing $\theta$-sequence of subsets of $A$, then there is some $j^*<\theta$ such that

\begin{equation*}
j^*\leq j<\theta\quad\implies\quad B_j=_I \bigcup_{i<\theta}B_i.
\end{equation*}
\item Whenever $\langle A_i : i<\theta\rangle$ is a sequence of $I$-positive sets, there is some $H\in[\theta]^\theta$ such that 

\begin{equation*}
\bigcap_{i\in H}A_i\neq\emptyset.
\end{equation*}
\end{enumerate}
\end{lemma}

At this point we can isolate one of the properties needed to push the generalized trichotomy theorem through.

\begin{definition}
Let $I$ be an ideal on a set $A$. For a regular cardinal $\theta$, we say that $I$ is $\theta$-regular if it satisfies one of the equivalent conditions in \cref{lem: reg}.
\end{definition}

Note that $I$ will automatically be $|A|^+$-regular. To see this, fix a sequence $\langle A_i : i<|A|^+\rangle$ of non-empty subsets of $A$. Then define a function $f:|A|^+\to A$ by setting $f(i)$ to be the least $a\in A$ such that $a\in A_i$. Then there must be some $H\subseteq |A|^+$ of cardinality $|A|^+$ and $a\in A$ such that $f(i)=a$ for every $i\in H$. In particular, if $I$ is an ideal on a set of ordinals $A$ and $|A|<\min(A)$, then $I$ will be $\theta$-regular for every $\theta\in (|A|, \min A]\cap \reg$.

\begin{definition}
Suppose that $I$ is an ideal on a set $A$. Let $\reg(I)$ denote the least regular $\theta$ such that $I$ is $\theta$-regular.
\end{definition}

For a set $A$ of ordinals, an ideal $I$ on $A$, and a regular cardinal $\theta$, recall that we are concerned with the following properties:

\begin{definition}
Let $F$ be a collection of functions from $A$ to $\mathrm{ON}$. We say that $f:A\to ON$ is an $I$-{\bf upper bound} for $F$ if $g\leq_I f$ for every $g\in F$. We say that $f$ is an $I$-{\bf least upper bound} for $F$ if additionally $f\leq_I f'$ for every upper bound $f'$ of $F$. Finally, $f$ is an $I$-{\bf exact upper bound} for $F$ if $f$ is a least upper bound of $F$, and $F$ is $<_I$-cofinal in $\{g\in{}^A\mathrm{ON}: g<_I f\}$. If the ideal $I$ is clear from the context, then it may be omitted.
\end{definition}

\begin{definition}

Let $\vec f=\langle f_\alpha : \alpha<\lambda\rangle$ be an $<_I$-increasing sequence of functions from $A$ to $\mathrm{ON}$. For a cardinal $\theta$, we define the following properties of $\vec f$:
\begin{enumerate}
\item \underline{$\mathrm{Good}_\theta$}: $\vec{f}$ has an exact upper bound $f\in {}^A\mathrm{ON}$ with $\{a\in A : \cf(f(a))<\theta\}\in I$.\\
\item \underline{$\mathrm{Bad}_\theta$}: There are sets $S(a)$ for each $a\in A$ such that $|S(a)|<\theta$ and an ultrafilter $D$ over $A$ disjoint from $I$ such that, for all $\xi<\lambda$, there exists some $h_\xi\in \prod_{a\in A}S(a)$ and some $\eta<\lambda$ such that $f_\xi<_D h_\xi<_D f_\eta$.\\
\item \underline{Ugly}: There is a function $g: A\to\mathrm{ON}$ such that, letting $t_\xi=\{a\in A : f_\xi(a)>g(a)\}$, the sequence $\vec{t}=\langle t_\xi : \xi<\lambda\rangle$ (which is $\subseteq_I$-increasing) does not stabilize modulo $I$. That is, for every $\xi<\lambda$, there is some $\xi<\eta<\lambda$ such that $t_\eta\setminus t_\xi\notin I$.\\
\end{enumerate}
\end{definition}

We begin by noting that the following two lemmas do not require any hypotheses on $I$ or $A$. The first of the two lemmas appears as Claim A.2 of \cite{Koj}

\begin{lemma}\label{lemma: lub to eub}
Suppose that $A$ is a set of ordinals, and $I$ is an ideal on $A$. Let $\lambda$ be regular and $\vec{f}=\langle f_\xi : \xi<\lambda\rangle$ be $<_I$-increasing. If $\vec{f}$ is not Ugly, then every least upper bound of $\vec{f}$ is an exact upper bound.
\end{lemma}

The next lemma appears in the middle of the proof of Theorem 2.15 of \cite{AbMag}, but we prove it for the sake of completeness.

\begin{lemma}\label{lemma: bad eub}
Suppose that $A$ is a set of ordinals, and $I$ is an ideal on $A$. Further, suppose that $\lambda$ and $\theta\leq\lambda$ are regular, and that $\vec{f}=\langle f_\xi : \xi<\lambda\rangle$ is a $<_I$-increasing sequence of functions from $A$ to $\mathrm{ON}$. If $\vec{f}$ has an exact upper bound $f$ such that $\{a\in A : cf(f(a))< \theta\}\notin I$ then $\vec{f}$ satisfies $\mathrm{Bad}_\theta$.
\end{lemma}

\begin{proof}
Let $f$ be an upper bound for $\vec{f}$ with $B=\{a\in A : cf(f(a))< \theta\}\notin I$, and let $D$ be an ultrafilter over $A$ disjoint from $I$ such that $B\in D$. Next for each $a\in B$, let $S(a)$ be cofinal in $f(a)$ with $|S(a)|=cf(f(a))<\theta$, and let $S(a)=\{0\}$ for each $a\notin B$. For each $\xi<\lambda$, let $f_\xi^+$ be defined by $f_\xi^+(a)=\min(S(a)\setminus f_\xi(a))$ for $i\in A$ and $f_\xi^+(a)=0$ otherwise.

Now for any $\xi<\lambda$ we have that $f_\xi<_D f_{\xi+1}^+$ where $f_{\xi+1}^+\in\prod_{a\in A}S(a)$. On the other hand, $f$ is exact and since $S(a)$ is cofinal in $f$ $D$-almost everywhere, it follows that there is some $\eta\in \lambda$ such that $f_{\xi+1}^+<_D f_\eta$ and so $\vec{f}$ is bad as witnessed by $\langle S(a) : a\in A \rangle$ and $D$.
\end{proof}

With these two lemmas in hand, we move to the statement and proof of the trichotomy theorem.

\begin{theorem}[Trichotomy]\label{thm: trichotomy}
Suppose that $A$ is a set of ordinals, and $I$ is an ideal on $A$. Let $\lambda>\reg(I)$ and $\theta\in[\reg(I),\lambda]$ be regular. If $\vec{f}=\langle f_\xi : \xi <\lambda\rangle$ is a $<_I$-increasing sequence of functions from $A$ to $\mathrm{ON}$, then at least one of $\mathrm{Good}_\theta$, $\mathrm{Bad}_\theta$, or $\mathrm{Ugly}$ must hold.
\end{theorem}

One thing to note is that the classical Trichotomy Theorem requires that $\lambda>|A|^+$, whereas we simply require that $\lambda>\reg(I)$. By work of Kojman and Shelah in \cite{ShKoj}, the requirement that $\lambda>|A|^+$ cannot be weakened to $\lambda>|A|$. 

\begin{proof}
We will show that, assuming $\vec{f}$ is not Ugly, then we can either find a witness to $\mathrm{Bad}_\theta$ or find a least upper bound $f$ for $\vec{f}$. By \cref{lemma: lub to eub}, this least upper bound is actually exact and so by \cref{lemma: bad eub} we can either find a witness to $\mathrm{Bad}_\theta$ or $f$ witnesses $\mathrm{Good}_\theta$. We proceed by induction on $\alpha<\reg(I)$, and at each stage create a candidate for a least upper bound. We will terminate at successor stages if we have found a least upper bound, and at limit stages if we can construct a witness to $\mathrm{Bad}_\theta$. At the end, we will show that we must have terminated at some $\alpha<\reg(I)$, else we will be able to derive a contradiction. At each stage $\alpha<\reg(I)$, we will define:

\begin{enumerate}
\item Functions $g_\alpha:A\to\mathrm{ON}$ which are $I$-upper bounds for $\vec{f}$ such that, for each $\beta<\alpha$, we have $g_\alpha\leq_I g_\beta$ but $g_\alpha\neq_I g_\beta$.
\item Sets $S^\alpha(a)=\{g_\beta(a) : \beta<\alpha\}$.
\item Functions $h^\alpha_\xi: A\to\mathrm{ON}$ for $\xi<\lambda$ defined by $h^\alpha_\xi(a)=\min(S^\alpha(a)\setminus f_\xi(a))$.\\
\end{enumerate}
Note here that 2) and 3) depend on how we define 1). Further, the sequence $\langle h_\xi^\alpha : \xi<\lambda\rangle$ is $\leq_I$-increasing in $\prod_{a\in A} S^\alpha(a)$.\\

\underline{Stage $\alpha=0$}: Here we let $g_0$ be any $\leq$-upper bound of $\vec{f}$, for example $g(a)=\sup\{f_\xi(a) : \xi<\lambda\}+1$ works. Requiring that $g_0$ dominates $\vec{f}$ everywhere ensures that the functions $h^\alpha_\xi$ are defined everywhere.\\

\underline{Stage $\alpha+1$}: Assume that $g_\alpha$ has been define. If $g_\alpha$ is a $I$-least upper bound for $\vec{f}$, then we can terminate the induction. Otherwise $g_\alpha$ is not a least upper bound, so there is some $I$-upper bound $g_{\alpha+1}$ such that $g_{\alpha+1}\leq_I g_\alpha$ but $g_{\alpha+1}\neq_I g_\alpha$.\\

\underline{Stage $\gamma$ limit}: Suppose that $\gamma<\reg(I)$ is a limit ordinal and that $g_\alpha$ has been defined for each $\alpha<\gamma$. Now consider the functions $h^\gamma_\xi$ and the sets
\begin{equation*}
t^\eta_\xi:=\{a\in A : h^\gamma_\xi(a) <f_\eta(a)\}
\end{equation*}
for $\eta, \xi<\lambda$. Fixing the $\xi$ coordinate, the function $h^\gamma_\xi$ is fixed while we run through $\langle f_\eta : \eta<\lambda\rangle$ and so the sequence $\vec{t}_\xi=\langle t^\eta_\xi : \eta<\lambda\rangle$ is $\subseteq_I$-increasing since the sequence $\vec{f}$ is $<_I$-increasing. Fixing the $\eta$ coordinate on the other hand, we fix $f_\eta$ and run through $\langle h^\gamma_\xi : \xi<\lambda\rangle$ and so the sequence $\vec{t}^\eta=\langle t^\eta_\xi : \xi<\lambda\rangle$ is $\subseteq_I$-decreasing. Since $\vec{f}$ is not Ugly, it follows that each sequence $\vec{t}_\xi$ stabilizes modulo $I$ at some ordinal $\eta(\xi)$. That is, for all $\eta>\eta(\xi)$, we have that $t^{\eta(\xi)}_\xi=_I t^\eta_\xi$. 

We have two cases to consider: either $t^{\eta(\xi)}_\xi\notin I$ for each $\xi< \lambda$, or for all sufficiently large $\xi<\lambda$, we have $t^{\eta(\xi)}_\xi\in I$. To see that these are indeed all of our cases, first note that whenever $\xi<\xi'$,
\begin{equation*}
t^{\eta(\xi')}_{\xi'}=_I t^\eta_{\xi'}\subseteq _I t^\eta_\xi=_I t^{\eta(\xi)}_\xi
\end{equation*}
for $\eta>\max\{\eta(\xi),\eta(\xi')\}$ since $\langle t^\eta_\xi : \xi<\lambda\rangle$ is $\subseteq_I$-decreasing. So if there is some $\xi\in \lambda$ for which $t^{\eta(\xi)}_\xi\in I$, then it follows that $t^{\eta(\xi')}_{\xi'}\in I$ for each $\xi'>\xi$ since $t^{\eta(\xi)}_{\xi}\in I$ and $t^{\eta(\xi')}_{\xi'}\subseteq_I t^{\eta(\xi)}_{\xi}$. 

Assume that the former happens (i.e. $t^{\eta(\xi)}_\xi\notin I$ for each $\xi< \lambda$), and consider the sequence $\langle t^{\eta(\xi)}_\xi:\xi<\lambda\rangle$. Note that this sequence is $\subseteq_I$-decreasing, and so $I^*\cup \{t^{\eta(\xi)}_\xi:\xi<\lambda\}$ has the finite intersection property. So let $D$ be an ultrafilter over $A$ extending $I^*\cup \{t^{\eta(\xi)}_\xi:\xi<\lambda\}$, and note that $\mathrm{Bad}_\theta$ is witnessed by $D$ and $\langle S^\gamma(a) : a\in A\rangle$. By construction we know that $f_\xi<_D h^\gamma_{\xi+1}$ for each $\xi<\lambda$. On the other hand, we have that $h^\gamma_{\xi+1}<_D f_{\eta(\xi+1)}$ since $t^{\eta(\xi+1)}_{\xi+1}\in D$. If this happens, we can terminate the induction.

Otherwise, suppose that $t^{\eta(\xi)}_\xi\in I$ for each sufficiently large $\xi<\lambda$. Let $\xi(\gamma)$ be the least $\xi$ for which this occurs, and define $g_\gamma:=h^\gamma_{\xi(\gamma)}$. Note then that $g_\gamma$ is an $I$-upper bound of $\vec{f}$ by construction, and so we only need to verify that $g_\gamma\leq_I g_\alpha$ while $g_\gamma\neq_I g_\alpha$ for each $\alpha<\gamma$. Recall that $S^\gamma(a)=\{g_\alpha (a) : \alpha<\gamma\}$ while $g_\gamma=\min(S^\gamma(a)\setminus f_{\xi(\gamma)}(a))$ and $\langle g_\alpha : \alpha<\gamma\rangle$ is $\leq_I$-decreasing. If $\alpha<\gamma$, then $\{ a\in A : g_\gamma(a)> g_\alpha(a)\}\in I$ since $g_\alpha$ is an $I$-upper bound for $f$, and thus $\{a\in A : g_\alpha(a)\in S^\gamma(a)\setminus f_{\xi(\gamma)}(a)\}\in I^*$. Additionally, it follows that $g_\gamma\neq_I g_\alpha$. Otherwise, for $\alpha<\beta<\gamma$, we get that
\begin{equation*}
g_\alpha=_I g_\gamma\leq_I g_\beta\leq_I g_\alpha,
\end{equation*}
and so $g_\alpha=_I g_\beta$. This contradicts condition (1) of the induction, and so $g_\gamma\neq_I g_\alpha$. It is worth noting that, in this case $\langle h^\gamma_\xi: \xi<\lambda\rangle$ stabilizes modulo $I$ again by definition.\\

We claim that this induction must have terminated. Otherwise, for each $\alpha\in \acc(\reg(I))$, we have defined: 

\begin{enumerate}
\item Functions $g_\alpha:A\to\mathrm{ON}$ which are upper bounds for $\vec{f}$ such that, for each $\beta\in \acc(\reg(I))$ with $\beta<\alpha$, we have $g_\alpha\leq_I g_\beta$ but $g_\alpha\neq_I g_\beta$.
\item Ordinals $\xi(\alpha)$ such that $g_\alpha=h^\alpha_{\xi(\alpha)}=_I h^\alpha_\xi$ for each $\xi\geq\xi(\alpha)$.\\
\end{enumerate}
Since $\reg(I)<\lambda$ with $\lambda$ regular, we can see that $\xi(*)=\sup\{\xi(\alpha): \alpha\in \acc(\reg(I))\}$ is still below $\lambda$. Note that $g_\alpha=_I h^\alpha_{\xi(*)}$ for each $\alpha\in \acc(\reg(I))$, so letting $H_\alpha=h^\alpha_{\xi(*)}$ we have that $H_\alpha$ enjoys the same properties as $g_\alpha$. Now for each $\alpha\in \acc(\reg(I))$, let $\alpha'$ be the successor of $\alpha$ in $\acc(\reg(I))$, i.e.
\begin{equation*}
\alpha'=\min(\acc(\reg(I))\setminus \alpha+1)=\alpha+\omega.
\end{equation*}

Define the sets
\begin{equation*}
B_\alpha:=\{a\in A : H_{\alpha'}(a)<H_\alpha(a)\}
\end{equation*}  
for each $\alpha\in \acc(\reg(I))$. Now, for all $\alpha<\beta\in \acc(\reg(I))$, we have $S^\alpha(a)\subseteq S^\beta(a)$ and hence $H_\beta\leq H_\alpha$. On the other hand, by construction each $B_\alpha\notin I$, and so the sequence $\langle B_\alpha :\alpha <\acc(\reg(I))\rangle$ has the property that, for some $H\subseteq \acc(\reg(I))$ with $|H|=\reg(I)$, the intersection $\bigcap_{\alpha\in H} B_\alpha$ is non-empty. Letting $a$ be in this intersection, we see that for all $\alpha<\beta\in H$:
\begin{equation*}
H_\beta(a)\leq H_{\alpha'}(a)< H_\alpha(a).
\end{equation*}

Thus, we have an infinite descending sequence of ordinals, which is a contradiction. Therefore the induction must have terminated and the theorem follows.
\end{proof}

One thing to note is that the trichotomy theorem above is indeed a generalization of the classical trichotomy theorem. This follows from the discussion above showing that any ideal $I$ on a set $A$ is automatically $|A|^+$-regular. Our next goal is to show that, under the same assumptions as the above trichotomy theorem, one can actually build Good exact upper bounds. In other words, we need to show that it is still possible to produce sequences $\vec f$ in $\prod A/I$ which are neither Bad nor Ugly. 

Implicit in \cite{Sh} is the fact that one can manufacture sequences with a property that is referred to as $(*)_\theta$ in \cite{AbMag}, and furthermore this property is equivalent to $\mathrm{Good}_\theta$. The property $(*)_\theta$ and this equivalence are used to show that any Good eub will have an stationary set of good points. Unfortunately for us, it is not clear whether or not sequences satisfy $(*)_\theta$ if and only if they satisfy $\mathrm{Good}_\theta$. 

Fortunately for us, we can show that if a sequence satisfies $(*)_\theta$, then it satisfies $\mathrm{Good}_\theta$. Further, we can produce sequences which satisfy $(*)_\theta$ directly. Throughout, we will fix a set of ordinals $A$, and an ideal $I$ on $A$ such that $\reg(I)<\min (A)$.

\begin{definition}
Let $X$ be a set of ordinals, and let $\vec{f}=\langle f_\xi : \xi\in X\rangle$ be a $<_I$-increasing sequence of functions from $A$ to $\mathrm{ON}$. We say that $\vec{f}$ is strongly increasing if there are sets $Z_\xi\in I$ for each $\xi\in X$ such that, for any $\eta<\xi\in X$, we have that $f_\eta(a)<f_\xi(a)$ for all $a\in A\setminus( Z_\eta\cup Z_\xi)$.
\end{definition}

The idea behind strongly increasing sequences is that the sets $Z_\xi$ serve as canonical witnesses that the sequence is $<_I$-increasing. 

\begin{definition}
Let $\lambda$ be a regular cardinal, and let $\vec{f}=\langle f_\xi : \xi<\lambda\rangle$ be a $<_I$-increasing sequence of functions from $A$ to $\mathrm{ON}$. Letting $\theta\leq\lambda$ be a regular cardinal, we say that $\vec{f}$ satisfies $(*)_\theta$ if for every $X\subseteq \lambda$ unbounded in $\lambda$, there exists a set $X_0\subseteq \lambda$ of size $\theta$ such that $\langle f_\xi : \xi\in X_0\rangle$ is strongly increasing.
\end{definition}

We should note that satisfying $(*)_\theta$ for a sequence of functions is somewhat analogous to satisfying $\reg(I)\leq\theta$ for an ideal $I$. 

\begin{lemma}\label{lemma: not ugly}
Let $\lambda>\reg(I)$ be a regular cardinal, $\vec{f}=\langle f_\xi : \xi<\lambda\rangle$ be a $<_I$-increasing sequence of functions from $A$ to $\mathrm{ON}$, and $\theta\in[\reg(I),\lambda]\cap\mathrm{REG}$. If $\vec{f}$ satisfies $(*)_\theta$, then $\vec{f}$ is not $\mathrm{Ugly}$.
\end{lemma}

\begin{proof}
Suppose otherwise, and let $g:A\to\mathrm{ON}$ witness that $\vec{f}$ is Ugly. That is, letting $t_\xi=\{a\in A : f_\xi(a)> g(a)\}$ for each $\xi<\lambda$, the sequence $\vec{t}=\langle t_\xi : \xi<\lambda\rangle$ does not stabilize modulo $I$. So for each $\xi<\lambda$, there is some $\eta>\xi$ such that $t_\eta\setminus t_\xi\notin I$. Using this, we can find an unbounded $X\subseteq \lambda$ such that, for all $\xi,\eta\in X$ with $\xi<\eta$, we have $t_\eta\setminus t_\xi\notin I$. Next, we use the fact that $\vec{f}$ satisfies $(*)_\theta$ to fix a set $X_0\subseteq X$ of size $\theta$ such that $\langle f_\xi : \xi\in X_0\rangle$ is strongly increasing as witnessed by $Z_\xi$ for each $\xi\in X_0$. 

For each $\xi\in X_0$, let $\xi'=\min(X_0\setminus(\xi+1))$ be the successor of $\xi$ in $X_0$, and let 
\begin{equation*}
A_\xi=(t_{\xi'}\setminus t_\xi)\cap (A\setminus(Z_{\xi'}\cup Z_\xi)).
\end{equation*}
Note that $A_\xi\notin I$ for each $\xi\in X_0$, and so we can find some $H\subseteq X_0$ of size $\reg(I)\leq\theta$ such that $\bigcap_{\xi\in H}A_\xi\neq\emptyset$. Let $a$ be in this intersection, and let $\xi,\eta\in H$ with $\xi<\eta$. Then we have that

\begin{equation*}
g(a)\geq f_\eta(a)\geq f_{\xi'}(a)>g(a).
\end{equation*}
Note that we get the first inequality from the fact that $a\notin t_{\eta}$, while the second inequality comes from the fact that $a\notin Z_\eta\cup Z_{\xi'}$ with $\eta\geq\xi'>\xi$, and the final inequality comes from the fact that $a\in t_{\xi'}$. This gives us that $g(a)>g(a)$, which is of course a contradiction.
\end{proof}

\begin{lemma}\label{lemma: not bad}
Let $\lambda>\reg(I)$ be a regular cardinal, $\vec{f}=\langle f_\xi : \xi<\lambda\rangle$ be a $<_I$-increasing sequence of functions from $A$ to $\mathrm{ON}$, and let $\theta\leq \lambda$ be regular such that $I$ is $\theta$-regular. If $\vec{f}$ satisfies $(*)_\theta$, then $\vec{f}$ is not $\mathrm{Bad}_\theta$.
\end{lemma}

\begin{proof}
Suppose otherwise, and let $S=\langle S(a): a\in A\rangle$ and $D$ witness that $\mathrm{Bad}_\theta$ holds. Let $X\subseteq \lambda$ be unbounded such that for all $\xi,\eta\in X$ with $\xi<\eta$, there is a function $h_\xi\in\prod_{a\in A}S(a)$ such that $f_\xi<_D h_\xi<_D f_\eta$. Using the fact that $\vec{f}$ satisfies $(*)_\theta$, let $X_0\subseteq X$ be of size $\theta$ such that $\langle f_\xi : \xi\in X_0\rangle$ is strongly increasing as witnessed by $Z_\xi\in I$ for each $\xi\in X_0$.

As before, for each $\xi\in X_0$, we let $\xi'=\min(X_0\setminus(\xi+1))$ be the successor of $\xi$ in $X_0$. For each $\xi\in X_0$, let $B_\xi=\{a \in A: f_\xi(a)<h_\xi(a) < f_{\xi'}(a)\}$ and define
\begin{equation*}
A_\xi= B_\xi\cap(A\setminus(Z_{\xi'}\cup Z_\xi)).
\end{equation*}
Note that $B_\xi\in D$, and so each $A_\xi$ is $I$-positive. So we can find some $H\subseteq X_0$ of size $\theta$ such that $\bigcap_{\xi\in H}A_\xi\neq\emptyset$, so let $a$ be in this intersection. Then for every $\xi, \eta\in X$ with $\xi<\eta$, we have that
\begin{equation*}
h_\xi(a)<f_{\xi'}(a)\leq f_{\eta}(a)<h_{\eta}(a).
\end{equation*}
The first inequality follows from $a\in B_\xi$, while the third follows from the fact that $a\in B_\eta$. The second inequality comes from the fact that $\xi'\leq\eta$ and $a\notin Z_{\xi'}\cup Z_\eta$. But then the sequence $\langle h_\xi(a) : \xi\in X_0\rangle$ is strictly increasing along $S(a)$ while $|S(a)|<\theta=|X_0|$ which is absurd.
\end{proof}

So for our purposes, it suffices to be able to construct sequences satisfying $(*)_\theta$ for appropriate $\theta$. We now quote Lemma 2.19 from \cite{AbMag}, which gives us conditions for constructing such sequences.

\begin{lemma}\label{lemma: star sequences}
Suppose that

\begin{enumerate}
\item $I$ is an ideal over $A$;
\item $\theta$ and $\lambda$ are regular cardinals such that $\theta^{++}<\lambda$;
\item $\vec{f}=\langle f_\xi : \xi<\lambda\rangle$ is a $<_I$-increasing sequence of functions from $A$ to $\mathrm{ON}$ such that for every $\delta\in S^{\lambda}_{\theta^{++}}$, there is a club $E_\delta\subseteq \delta$ such that for some $\delta\leq\delta'<\lambda$, we have 

\begin{equation*}
\sup\{f_\alpha : \alpha\in E_\delta\}<_I f_{\delta'};
\end{equation*}
Then $(*)_\theta$ holds for $\vec{f}$.
\end{enumerate}
\end{lemma}

It turns out that, while the above lemma looks technical, constructing sequences with the above properties is itself easy. The proof of the following theorem is identical to the proof of Theorem 2.21 of \cite{AbMag}, but we include it for the sake of completeness.

\begin{theorem}\label{thm: good sequences}
Suppose that $A$ is a set of regular cardinals. Let $\lambda>\reg(I)$ be a regular cardinal such that $\prod A/I$ is $\lambda$-directed, and let $\vec{f}=\langle f_\xi : \xi<\lambda\rangle$ be any $<_I$-increasing sequence of functions in $\prod A$. Then there exists a sequence $\vec{g}=\langle g_\xi : \xi<\lambda\rangle$ such that: 

\begin{enumerate}
\item $\vec{g}$ is $<_I$-increasing;
\item for each $\xi<\lambda$, we have $f_\xi<g_{\xi+1}$;
\item for every $\theta<\lambda$ regular such that $\theta^{++}<\lambda$, $\{a\in A : a\leq \theta^{++}\}\in I$, and $I$ is $\theta$-regular, we have that $\vec{g}$ is $\mathrm{Good}_\theta$.
\end{enumerate}
\end{theorem}

\begin{proof}
By \cref{lemma: not ugly} and \cref{lemma: not bad}, it suffices to produce a sequence which satisfies $(*)_\theta$ for every appropriate $\theta$. In other words, we only need to produce a sequence satisfying the last condition in \cref{lemma: star sequences}. We proceed by induction on $\xi<\lambda$.

At stage $0$, we simply let $g_0$ be any function in $\prod A/I$. At successor stages, suppose that $g_\xi$ has been defined and let $g_{\xi+1}$ be defined by 
\begin{equation*}
g_{\xi+1}(a)=\max\{g_\xi(a), f_\xi(a)\}+1.
\end{equation*}
At limit stages $\delta$, we have two cases to deal with. In the first case, we suppose that $cf(\delta)=\theta^{++}$ for $\theta$ as in condition $(3)$, and let $E_\delta\subseteq \delta$ be club of order type $\theta^{++}$. Define 
\begin{equation*}
g_\delta=\sup\{g_\xi : \xi\in E_\delta\},
\end{equation*}
and note that $g_\delta(a)<a$ whenever $a>\theta^{++}$ and so $g_\delta\in\prod A/I$. In the other case, simply let $g_\delta'$ be a $\leq_I$-upper bound of $\{g_\xi : \xi<\delta\}$ and set $g_\delta=g_\delta'+1$. 

By construction, the sequence $\vec{g}$ satisfies the hypotheses of \cref{lemma: star sequences}, and so we are finished.
\end{proof}
\section{Generators for $\lambda\in\pcf_I(A)$}

In the classical pcf theory, the trichotomy theorem is used to produce generators for every $\lambda\in \pcf(A)$. We would like to do exactly that for each $\lambda \in\pcf_I(A)$ when $A$ is a set of regular cardinals, and $I$ is an ideal on $A$ satisfying $\reg(I)<\min(A)$. As noted earlier, Shelah does obtain generators for $\pcf_I(A)$ in \cite{Sh506}. The benefit of our approach is that the exposition has been streamlined to mimic the modern development of pcf theory as found in \cite{AbMag}. The results in the section are due to Shelah unless otherwise noted.

\begin{definition}
Suppose that $(P, <)$ is a partial order. We say that $P$ has true cofinality $\lambda$ if there is a $<$-linearly ordered family $F\subseteq P$ of cofinality $\lambda$ which is itself cofinal in $(P,<)$. In this case, we write $\tcf(P,<)=\lambda$, though the ordering $<$ may be omitted if it is clear from the context.
\end{definition}

We should note that $\tcf(P,<)$ may not always be defined, as $(P,<)$ could very well not have a linearly ordered cofinal subset. When it is defined, it is always a regular cardinal and $\tcf(P,<)=\cf(P,<)$. 

\begin{definition}
Suppose $A$ is a collection of ordinals and $I$ is a fixed ideal over $A$. Define

\begin{equation*}
\pcf_I(A)=\{\cf(\prod A/D) : D\text{ is an ultrafilter over }A\text{ disjoint from }I \}.
\end{equation*}.
\end{definition}

As $\prod A/D$ is linearly ordered, it follows that every element of $\pcf_I(A)$ is a regular cardinal.

\begin{definition}
Suppose $A$ is a collection of ordinals and $I$ is a fixed ideal over $A$. The ideal $J_{<\lambda}^I[A]$ is defined to be the collection of sets $B\subseteq A$ satisfying:

\begin{enumerate}
\item $B\in I$ or;
\item for every ultrafilter $D$ over $A$ disjoint from $I$, if $B\in D$, then $\cf(\prod A/D)<\lambda$.
\end{enumerate}
We will denote these ideals by $J_{<\lambda}^I$ if the set $A$ is clear from context.
\end{definition}

We now highlight a number of simple properties of $\pcf_I(A)$ and $J_{<\lambda}^I[A]$. 

\begin{lemma}\label{lem: properties}
Suppose $A$ is a collection of ordinals and $I$ is a fixed ideal over $A$. 

\begin{enumerate}
\item If $\lambda\in \pcf_I(A)$, then $J_{<\lambda}^I$ is proper.
\item If $\lambda\notin\pcf_I(A)$, then $J_{<\lambda}^I=J_{<\lambda^+}^I$.
\item If $F$ is a filter disjoint from $I$ with $\lambda=\tcf(\prod A/F)$, then $\lambda\in\pcf_I(A)$.
\item For $B\subseteq A$ which is $I$-positive, if we set $I_B:=\mathcal{P}(B)\cap I$, then $\pcf_{I_B}(B)\subseteq\pcf_I(A)$.
\item If $I$ and $J$ are ideals over $A$ such that $I\subseteq J$, then $\pcf_J(A)\subseteq\pcf_I(A)$.
\item If $B\subseteq A$ is such that $B=_I A$, then $\pcf_{I_B}(B)=\pcf_I(A)$.
\end{enumerate}
\end{lemma}

\begin{proof}
As the proofs of these results are routine, we will content ourselves with only proving $(6)$, and leaving the rest for the reader.

(6): We already know that $\pcf_{I_B}(B)\subseteq\pcf_I(A)$ from (4), so it suffices to show the other direction. With that in mind, let $\lambda\in\pcf_I(A)$ and let $D$ be an ultrafilter over $A$ disjoint from $I$ such that $\cf(\prod A/D)=\lambda$. Note that $B\in D$ so we can define $D_B=\mathcal{P}(B)\cap D$ which is an ultrafilter over $B$ extending $I^*_B$. Now if we let $\vec{f}=\langle f_\xi : \xi<\lambda\rangle$ be cofinal in $\prod A/D$, then we see that $\vec{f}\upharpoonright B=\langle f_\xi\upharpoonright B: \xi<\lambda\rangle$ is cofinal in $\prod B/D_B$. Thus, $\lambda\in \pcf_{I_B}(B)$.
\end{proof}

With (4) and (6) above in mind, we set aside some notation

\begin{definition}
Suppose $A$ is a collection of ordinals, and $I$ is some ideal over $A$. For any $I$-positive $B\subseteq A$, we define 
\begin{equation*}
\pcf_I(B):=\pcf_{I_B}(B)
\end{equation*}
\end{definition}

We first show that only assuming $\wsat(I)<\min(A)$ is enough to get $\lambda$-directedness of $J_{<\lambda}^I[A]$ whenever it is proper. The following appears as Lemma 1.9 in \cite{Sh506}, but we include a proof for the sake of completeness.

\begin{lemma}\label{lemma: directedness}
Suppose that $A$ is a collection of regular cardinals with no maximum, and that $I$ is an ideal over $A$ such that $\min(A)>\wsat(I)$. If $\lambda\geq \wsat (I)$ is a cardinal with $J_{<\lambda}^I[A]$ proper, then $\prod A/J_{<\lambda}^I$ is $\lambda$-directed.
\end{lemma}

\begin{proof}
We will show by induction on $\lambda_0 <\lambda$ that $\prod A/J_{<\lambda}^I$ is $\lambda_0^+$ directed. If $F\subseteq\prod A$ is such that $|F|\leq \wsat(I)<\min(A)$, then we let $g$ be defined by $g(a)=\sup\{f(a) : f\in F\}$. Then since each $a\in A$ is regular, it follows that $g\in \prod A$ and $f\leq g$ everywhere.

By way of induction, assume we have shown for some cardinal $\lambda_0$ with $\wsat(I)<\lambda_0<\lambda$, that $\prod A/J_{<\lambda}^I$ is $\lambda_0$-directed, and let $F\subseteq \prod A$ of size $\lambda_0$ be given. We first assume that $\lambda_0$ is singular. In this case, we can write $F=\bigcup_{\alpha<cf(\lambda_0)}F_\alpha$ such that $|F_\alpha|<\lambda_0$. Then by assumption, we can bound each $F_\alpha$ by some $g_\alpha$, and then bound the set $\{ g_\alpha : \alpha<cf(\lambda_0)\}$ by some $g\in \prod A$. We then have that $f\leq g$ modulo $J_{<\lambda}^I$ for each $f\in F$.

So assume that $\lambda_0$ is regular. We begin by replacing $F=\{h_i : i<\lambda_0\}$ with a $\leq_{J_{<\lambda}^I}$-increasing sequence $\vec{f}=\langle f_i : i<\lambda_0\rangle$. We just let $f_i$ be a $\leq_{J_{<\lambda}^I}$-upper bound for $\{h_j : j\leq i\}\cup \{f_j : j<i\}$. By construction, if we can find a $g\in \prod A$ such that $f_i \leq g$ modulo $J_{<\lambda}^I$ for each $i<\lambda_0$, then we will be done. At this point, we will proceed by induction on $\alpha <\wsat(I)$ and attempt to construct a $\leq_{J_{<\lambda}^I}$-increasing sequence of candidates for bounds of $\vec{f}$. As usual, we will show that this construction must terminate at some point, or we will be able to generate a contradiction.

By induction on $\alpha <\wsat(I)$, we will define functions $g_\alpha$, ordinals $\xi(\alpha)$, and sequences $\langle B^\alpha_\xi : \xi <\lambda_0\rangle$ with the following properties:

\begin{enumerate}
\item $g_\alpha\in \prod A$ and for all $\beta<\alpha$, we have that $g_\beta\leq g_\alpha$;
\item $B_\xi^\alpha:=\{a\in A : f_\xi(a)>g_\alpha (a)\}$;
\item For each $\alpha <\wsat(I)$, and every $\xi\in[\xi(\alpha+1),\lambda_0)$, we have that $B^\alpha_\xi\neq B^{\alpha+1}_\xi$ modulo $J_{<\lambda}^I$.
\end{enumerate}

The construction proceeds as follows. At stage $\alpha=0$, we simply let $g_0=f_0$, and set $\xi(\alpha)=0$ (note that $\xi(\alpha)$ only matters when $\alpha=\beta+1$ for some ordinal $\beta<\wsat(I)$). At limit stages, assume that $g_\beta$ has been defined for each $\beta<\alpha$, and define $g_\alpha$ by setting $g_\alpha(a)=\sup_{\beta<\alpha} g_\beta(a)$. Note since $\alpha<\wsat(I)<\min(A)$ and each $a\in A$ is regular, that $g_\alpha\in\prod A$.

At successor stages, let $\alpha=\beta+1$, and suppose that $g_\beta$ has been defined. If $g_\beta$ is a $\leq_{J_{<\lambda}^I}$-upper bound for $\vec{f}$, then we're done and we can terminate the induction. Otherwise, note that the sequence $\langle B^\beta_\xi : \xi <\lambda_0\rangle$ is $\subseteq_{J_{<\lambda}^I}$-increasing and so there is a minimum $\xi(\alpha)$ for which every $\xi\in[\xi(\alpha),\lambda_0)$ has the property that $B_\xi^\beta\notin J_{<\lambda}^I$ (else if there is no such $\xi(\alpha)$, then $g_\beta$ was indeed the desired bound). By definition, that means we can find some ultrafilter $D$, disjoint from $J_{<\lambda}^I$ such that $B_{\xi(\alpha)}^\beta\in D$ and $\cf(\prod A/D)\geq\lambda$. Thus it follows that $\vec{f}$ must have a $<_D$-upper bound in $\prod A$, say $\hat{g}_\alpha$. We then define $g_\alpha\in \prod A$ by $g_\alpha (a)=\max\{g_\beta(a), \hat{g}_\alpha(a)\}$.

Note that for each $\xi\in[\xi(\beta+1),\lambda_0)$, we have that $B_\xi^\beta\in D$. On the other hand, our definition of $g_\alpha$ gives us that $B_\xi^{\beta+1}\notin D$ since $g_\alpha$ is at least $\hat{g}_\alpha$ everywhere. Thus, condition 3) is satisfied, as are 1) and 2) trivially by construction.

We claim that this process must have terminated at some stage. Otherwise, we let $\xi(*)=\sup\{\xi(\alpha) : \alpha<\wsat(I)\}$, and note that each $B^\alpha_{\xi(*)}\notin J_{<\lambda}^I$ since the induction never terminated. Next, we note that conditions 1) and 3) give us that for $\alpha\leq \beta$, we have $B_{\xi(*)}^\beta\subseteq B_{\xi(*)}^\alpha$ and so $B_{\xi(*)}^\alpha\setminus B_{\xi(*)}^{\alpha+1}\notin J_{<\lambda}^I$. Therefore, for $\alpha<\beta$, we have that $B_{\xi(*)}^\beta\subseteq B_{\xi(*)}^{\alpha+1}$ and so the sets $B_{\xi(*)}^\alpha\setminus B_{\xi(*)}^{\alpha+1}$ and $B_{\xi(*)}^\beta\setminus B_{\xi(*)}^{\beta+1}$ are disjoint $I$-positive sets (since $J_{<\lambda}^I$ extends $I$). But then we have a partition $\{B_{\xi(*)}^\alpha\setminus B_{\xi(*)}^{\alpha+1} : \alpha < \wsat(I)\} $ of A into $\wsat(I)$-many disjoint $I$-positive sets, which is a contradiction. Therefore the process terminated at some point and $\vec{f}$ (hence $F$) has a $\leq_{J_{<\lambda}^I}$-upper bound. This completes the induction and the proof.

\end{proof}

Throughout the remainder of this section, we fix a set of regular cardinals $A$ with no maximum, and an ideal $I$. In line with the notation of \cite{AbMag}, we will isolate the additional hypothesis of \cref{lemma: directedness}.

\begin{definition}
We say that $A$ is weakly progressive over $I$ if $\wsat(I)<\min(A)$. We say that $A$ is progressive over $I$ if additionally, $\reg(I)<\min (A)$.
\end{definition}

From $\lambda$-directedness, we immediately recover the following facts. Note that the proofs of the following two corollaries only utilize $\lambda$-directedness, and can be found in \cite{AbMag} as Corollary 3.5 and Corollary 3.7 respectively.

\begin{corollary}
If $A$ is weakly progressive over $I$, then for every ultrafilter $D$ over A disjoint from $I$, $\cf(\prod A/D)\geq\lambda$ if and only if $J_{<\lambda}^I[A]\cap D=\emptyset$.
\end{corollary}

\begin{corollary}\label{corollary: maxpcf}
If $A$ is weakly progressive over $I$, then $\max\pcf_I(A)$ exists.
\end{corollary}

As we are aiming to obtain generators using the trichotomy theorem, our next natural step is to show that we can get universal cofinal sequences.

\begin{definition}
Suppose that $\lambda\in\pcf_I(A)$. A sequence $\vec f=\langle f_\xi : \xi<\lambda\rangle$ of functions in $\prod A$ is a universal cofinal sequence for $\lambda$ if and only if

\begin{enumerate}
\item $\vec f$ is $<_{J_{<\lambda}^I}$-increasing.
\item For every ultrafilter $D$ over $A$ disjoint from $I$ such that $\lambda=\cf(\prod A/D)$, $\vec{f}$ is cofinal in $\prod A/D$.
\end{enumerate}
\end{definition}

\begin{theorem}\label{thm: universal sequences}
If $A$ is weakly progressive over $I$, then every $\lambda\in \pcf_I(A)$ has a universal cofinal sequence.
\end{theorem}

\begin{proof}

The proof of this will be very similar to the proof of \cref{lemma: directedness}, insofar as we will proceed by induction on $\alpha<\wsat(I)$, and suppose that we fail to get a universal cofinal sequence at each stage. From this we will be able to produce a contradiction to weak saturation.

We will proceed by induction on $\alpha<\wsat(I)$, and construct candidate universal sequences $\vec{f}^\alpha=\langle f_\xi^\alpha : \xi<\lambda\rangle$. Now, we want to come up with sets $B^\alpha_\xi$ that are $\subseteq$-increasing in the $\alpha$ coordinate but differ from each other modulo $J_{<\lambda}^I$ (and hence $I$). So we will ask that not only is the collection $\langle f^\alpha_\xi : \alpha<\wsat(I),~ \xi<\lambda\rangle$ strictly increasing modulo $J_{<\lambda}^I$ in the $\xi$ coordinate, but that it is $\leq$-increasing in the $\alpha$ coordinate. With that in mind, we will use $\lambda$-directedness to inductively construct these sequences.

At stage $\alpha=0$, we let $\vec{f}^0=\langle f^0_\xi : \xi<\lambda\rangle$ be any $<_{J_{<\lambda}^I}$-increasing sequence in $\prod A$. We can create such a sequence inductively as follows: let $f^0_0$ be arbitrary, and then assume that $f^0_\eta$ has been defined for $\eta<\xi$. By $\lambda$-directedness, we can find $g\in\prod A$ such that $f_\eta^0\leq_{J_{<\lambda}^I} g$ for all $\eta<\xi$, and let $f^0_\xi=g+1$.

At limit stages, let $\gamma<\lambda$ and assume that $\vec{f}^\alpha$ has been defined for each $\alpha<\gamma$. We inductively define $\vec{f}^\gamma=\langle f^\gamma_\xi : \xi<\lambda\rangle$ as follows: let $f^\gamma_0=\sup \{f^\alpha _0 : \alpha<\gamma\}$, which is in $\prod A$ since $\gamma<\wsat(I)<\min(A)$. Now suppose that $f^\gamma_\eta$ has been defined for each $\eta<\xi$, and let $g=\sup \{f^\alpha_\eta : \alpha<\gamma\}$. Again $g\in \prod A$, and let $h$ be such that $f^\gamma_\eta\leq_{J_{<\lambda}^I} h$ for all $\eta<\xi$ by $\lambda$-directedness. Then define $f^\gamma_\xi$ by $f^\gamma_\xi(a)=\max\{g(a),h(a)\}+1$, which is as desired.

At successor stages suppose that $\vec{f}^\alpha$ has been defined. If $\vec{f}^\alpha$ is a universal sequence, then we can terminate the induction. If not, we inductively define $\vec{f}^{\alpha+1}=\langle f^{\alpha+1}_\xi: \xi<\lambda\rangle$ as follows: Since $\vec{f}^\alpha$ is not universal, we can find an ultrafilter $D_\alpha$ over $A$ with the property that $\cf(\prod A/D_\alpha)=\lambda$, but $\vec{f}^\alpha$ is $<_{D_\alpha}$-dominated by some $h\in\prod A/D_\alpha$ (note that $D_\alpha$ is disjoint from $J_{<\lambda}^I$). Let $\vec{g}=\langle g_\xi : \xi<\lambda\rangle$ be a $<_{D_\alpha}$-increasing, cofinal sequence in $\prod A/D_\alpha$. We define $f^{\alpha+1}_0$ by setting $f^{\alpha+1}_0(a)=\max\{h(a), f^{\alpha}_0(a), g_0(a)\}$. Now suppose that $f^{\alpha+1}_\eta$ has been defined for each $\eta<\xi$, and let $\hat{f}$ be such that $f^{\alpha+1}_\eta\leq_{J_{<\lambda}^I} \hat{f}$ for all $\eta<\xi$ by $\lambda$-directedness. Then define $f^{\alpha+1}_\xi$ by $f^{\alpha+1}_\xi(a)=\max\{f^\alpha_\xi(a),\hat{f}(a), g_\xi(a)\}+1$, which is as desired. Note that $\vec{f}^{\alpha+1}$ is cofinal in $\prod A/D_\alpha$

We claim that we must have terminated the induction at some stage. Otherwise, we will have defined for each $\alpha<\wsat(I)$ the following:

\begin{enumerate}
\item Sequences $\vec{f}^\alpha=\langle f^\alpha_\xi : \xi<\lambda\rangle$ which are $J_{<\lambda}^I$-increasing in the $\xi$ coordinate, and $\leq$-increasing in the $\alpha$ coordinate.
\item Ultrafilters $D_\alpha$ disjoint from $J_{<\lambda}^I$ such that $\vec{f}^\alpha$ is $<_{D_\alpha}$ dominated by $f^{\alpha +1}_0$, and $\vec{f}^{\alpha+1}$ is cofinal in $\prod A/D_\alpha$.
\end{enumerate}

We will use this to derive a contradiction. We begin by letting $h\in\prod A$ be defined by setting $h(a)=\sup\{f_0^\alpha (a) : \alpha<\wsat(I)\}$ (recall that $\wsat(I)<\min(A)$). By condition 2) above, for every $\alpha<\wsat(I)$, there exists an index $\xi(\alpha)<\lambda$ such that $h <_{D_\alpha} f_{\xi(\alpha)}^{\alpha +1}$. Since $\wsat(I)<\min(A)\leq\lambda$ for $\lambda$ regular, it follows that $\xi(*)=\sup\{\xi(\alpha) : \alpha<\wsat(I)\}$ is below $\lambda$. So, for each $\alpha<\wsat(I)$, we have that $h<_{D_\alpha} f_{\xi(*)}^{\alpha+1}$. Now define the sets
\begin{equation*}
B_\alpha =\{a\in A : h(a) \leq f^\alpha_{\xi(*)}(a)\}.
\end{equation*}

By construction, we have that $B_\alpha\notin D_\alpha$ since $f_{\xi(*)}^\alpha <_{D_\alpha} f_0^{\alpha+1}\leq h$. On the other hand, $B_{\alpha+1}\in D_\alpha$ since $h<_{D_\alpha} f^{\alpha+1}_{\xi(*)}$. So, it follows that $B_\alpha\neq B_{\alpha+1}$ modulo $J_{<\lambda}^I$ (hence modulo $I$). But since $f^\alpha_{\xi(*)}\leq f^{\alpha+1}_{\xi(*)}$, we have that $B_\alpha\subseteq B_{\alpha+1}$ (in fact $\beta<\alpha$ implies that $B_\beta\subseteq B_\alpha$) and so we are in the same position as the proof of \cref{lemma: directedness}. That is, $\langle B_{\alpha+1}\setminus B_{\alpha} : \alpha<\wsat(I)\rangle$ is a collection of $I$-positive sets which are disjoint, contradicting weak saturation. Therefore, the induction must have halted at some stage and we are done.

\end{proof}

Now that we have universal cofinal sequences, we can recover the following corollary by repeating the standard arguments (Theorem 4.4 from \cite{AbMag}).

\begin{corollary}\label{corollary: cof}
If $A$ is weakly progressive over $I$, then $cf(\prod A/I)=\max\pcf_I(A)$.
\end{corollary}

\begin{definition}
Let $\lambda\in \pcf_I(A)$. We say that $B$ is a generator of $J_{<\lambda^+}^I[A]$ (written $J_{<\lambda^+}^I[A]=J_{<\lambda}^I[A]+B$) if the ideal $J_{<\lambda^+}^I[A]$ is generated by $J_{<\lambda}^I[A]\cup \{B\}$.
\end{definition}

The pcf theorem (in the classical theory) is the statement that, for every $\lambda\in \pcf(A)$, we can find a generator. We now show how to extract generators from universal cofinal sequences under the appropriate conditions. As the proof of the following lemma can be recovered by repeating the standard arguments (as found in the beginning of the proof of e.g. Theorem 4.8 of \cite{AbMag}), we omit said proof.

\begin{lemma}\label{lemma: weak to pcf}
Suppose that $\lambda\in \pcf_I(A)$ has a universal cofinal sequence $\vec{f}=\langle f_\xi : \xi<\lambda\rangle$ with an exact upper bound $f$. Then the set $B=\{a\in A : f(a)=a\}$ is a generator for $J_{<\lambda^+}^I[A]$.
\end{lemma}

The following is immediate.

\begin{theorem}[The pcf Theorem]\label{thm: pcf theorem}
If $A$ is progressive over $I$, then for every $\lambda\in pcf_I(A)$, there exists a $B_\lambda\subseteq A$ such that $J_{<\lambda^+}^I[A]=J_{<\lambda}^I[A]+B_\lambda$.
\end{theorem}

\begin{proof}
Fix $\lambda\in\pcf_I(A)$. Note that if $\lambda\in \pcf_I(A)$ and $\lambda\leq\reg(I)^{++}$, then $\{\lambda\}\notin I$ and, desired generator is simply $\{\lambda\}$.  So, assume that $\reg(I)^{++}<\lambda$ and apply \cref{thm: universal sequences} to obtain a universal cofinal sequence $\vec f$ for $\lambda$. As $J:=J_{<\lambda}^I[A]$ is $\lambda$ directed and $\reg(I)^{++}<\lambda$, we may apply \cref{thm: good sequences} to obtain a $<_{J}$-increasing sequence $\vec{g}$ which pointwise dominates $\vec f$ such that $\vec{g}$ has an eub $f$. It is easily seen that any $<_{J}$-increasing sequence which $<_{J}$-dominates $\vec{f}$ is also universal for $\lambda$. Thus, we may apply \cref{lemma: weak to pcf} to $\vec {g}$ to obtain the desired generator 

\begin{equation*}
B=\{a\in A : f(a)=a\}.
\end{equation*}

\end{proof}

As an easy corollary, we can obtain the compactness theorem modulo $I$.

\begin{theorem}[Compactness]\label{thm: compactness}
Suppose that $A$ is progressive over $I$, and let $\langle B_\lambda : \lambda\in \pcf_I(A)\rangle$ be a sequence of generators. For any $I$-positive $X\subseteq A$, we can find $n<\omega$ and $\{\lambda_i : i\leq n\}\subseteq\pcf_{I}(X)$ such that
\begin{equation*}
X\subseteq_I \bigcup_{i\leq n}B_{\lambda_i}.
\end{equation*}
\end{theorem}

\begin{proof}
Recall that $I_X:=\mathcal{P}\cap I$. As $X$ is still progressive over $I_X$, it follows that $\lambda_0:=\max\pcf_{I}(X)$ exists, so let $X_1:=X\setminus B_{\lambda_0}$. If $X_1\in I$, then we have that $X=_I B_{\lambda_1}$. Otherwise, we note that $X_1$ is progressive over $I_{X_1}$ and so $\max\pcf_{I}(X_1)$ exists and is equal to some $\lambda_1<\lambda_0$. Continuing on in this manner, we will reach some finite stage $n<\omega$ such that $\max\pcf_I(X_n)=\lambda_n$ and $X_n\setminus B_{\lambda_n}$. At this point, we are done since

\begin{equation*}
X\subseteq_I \bigcup_{i\leq n}B_{\lambda_i}.
\end{equation*}
\end{proof}
\section{Obstacles and Questions}

With generators in hand, the natural thing to ask is whether or not one can obtain something like the no holes conclusion. That is, can we show that if $A$ is an interval of regular cardinals, then so is $\pcf_I(A)$? We should expect not, as $\pcf_I(A)$ only depends on $A$ modulo $I$, and in general we cannot expect $I^*$ to only concentrate on intervals of regular cardinals. 

\begin{lemma}
It is consistent that $A$ is an interval of regular cardinals, while $\pcf_I(A)$ fails to be.
\end{lemma}

\begin{proof}
For this, we work in a model where $\aleph_\omega$ is strong limit while $2^{\aleph_\omega}=\aleph_{\omega+4}$, and let $A=[\aleph_2,\aleph_\omega)\cap\mathrm{REG}$. Note in this case that any ideal $I$ over $A$ will be $\theta$-regular for every regular $\theta\geq \aleph_1$. Further, we have that $\pcf(A)=[\aleph_2,\aleph_{\omega+4}]\cap\mathrm{REG}$ by the classical pcf theory, and in particular we have generators for each $\lambda\in \pcf(A)$. So we let $B_i=B_{\aleph_{\omega+i}}$ for each $1\leq 1\leq 4$, noting that we may assume that the sets $B_i$ are disjoint since $B_i$ is only unique modulo $J_{<\aleph_{\omega+i}}[A]$. Finally, let $B=B_1\cup B_3$ and let $I$ be the ideal over $A$ defined by

\begin{equation*}
X\in I \iff |X\setminus B|<\aleph_0.
\end{equation*}
\\

We claim that $\pcf_I(A)=\{\aleph_{\omega+2},\aleph_{\omega+4}\}$. We begin by noting that each $B_i$ is unbounded in $A$, and so $I$ extends the ideal of bounded sets. For each $1<n<\omega$, we have that $J_{<\aleph_n}[A]=\mathcal{P}(\{\aleph_2,\ldots, \aleph_{n-1}\})$ and so $\{\aleph_n\}$ is a generator for $J_{<\aleph_{n+1}}[A]$. Therefore, $B_{\lambda}\in I$ for each $\lambda\in\pcf(A)\setminus\{\aleph_{\omega+2},\aleph_{\omega+4}\}$. For such a $\lambda$, let $\vec{f}=\langle f_\xi : \xi<\lambda\rangle$ be a universal cofinal sequence for $\lambda$ with exact upper bound $f\in{}^A\mathrm{ON}$ such that $B_\lambda=\{a\in A : f(a)=a\}$ (recall that this is how we obtain generators in the first place). In this case, $f+1\in\prod A/I$ and so $\vec{f}$ has a $<_I$-upper bound. As $\vec{f}$ is universal, it follows that there is no ultrafilter $D$ over $A$ extending  the dual of $I$ with $cf(\prod A/D)=\lambda$.

Now we only need to show that $\aleph_{\omega+2},\aleph_{\omega+4}\in\pcf_I(A)$ and we are done. For this, simply note that $B_i$ is $I$-positive for $i=2$ or $4$ and so we can find an ultrafilter $D$ over $A$ disjoint from $I$ containing $B_i$. Let $\vec{f}=\langle f_\xi : \xi<\aleph_{\omega+i}\rangle$ be a universal cofinal sequence for $\aleph_{\omega+1}$ with exact upper bound $f\in{}^A\mathrm{ON}$ such that $B_i=\{a\in A : f(a)=a\}$. Then $\vec{f}$ is cofinal in $\prod A/D$, which means that $\aleph_{\omega+i}\in\pcf_I(A)$ for $i=2$ or $4$. 
\end{proof}

The next thing to note is that, even though the proofs may be different, we obtained generators for $\lambda\in\pcf_I(A)$ by generalizing the standard techniques. So, we might ask if it is possible to employ this strategy to obtain transitive generators.

\begin{definition}
Suppose that $A$ is a set of regular cardinals, and $A\subseteq N\subseteq \pcf(A)$ is such that $N$ carries a generating sequence $B=\langle B_\lambda : \lambda\in N\rangle$. We say that $B$ is transitive if for every $\lambda\in N$, if $\theta\in B_\lambda$, then $B_\theta\subseteq B_\lambda$.
\end{definition}

Unfortunately, there are a number of obstacles to obtaining transitive generators through this route. In order to explain precisely what these obstacles are, we need several tools involving elementary submodels. Following standard abuse of notation, we will use $H(\chi)$  to refer to the structure $(H(\chi),\in,<_\chi)$ where $<_\chi$ well orders $H(\chi)$.

\begin{definition}
Suppose $\kappa<\chi$ are regular cardinals. We say $N\prec H(\chi)$ is a $\kappa$-presentable substructure if $N=\bigcup_{\alpha<\kappa}N_\alpha$ where

\begin{enumerate}
\item $N_\alpha\prec H(\chi)$ for each $\alpha<\kappa$;
\item $\langle N_\alpha : \alpha<\kappa\rangle$ is $\subseteq$-increasing and continuous;
\item $N_\alpha\in N_{\alpha+1}$ for each $\alpha<\kappa$;
\item $\kappa+1\subseteq N_0$;
\item $|N_\alpha|=\kappa$ for each $\alpha<\kappa$.
\end{enumerate}
\end{definition}

\begin{definition}
For any structure $N$, we let $Ch_N$ denote the characteristic function of $N$, defined by setting
\begin{equation*}
\Ch_N(\mu)=\sup N\cap\mu,
\end{equation*}
where $\mu$ is a regular cardinal. Note that if $|N|<\mu$, then $\Ch_N(\mu)\in\mu$.
\end{definition}

Let $A$ be a set of regular cardinals with $|A|^+<\min(A)$, and fix a sufficiently large and regular $\chi$. Let $N\prec H(\chi)$ be a $\kappa$-presentable structure with $N=\bigcup_{\alpha<\kappa}N_\alpha$, where $|A|<\kappa<\min(A)$.  The arguments for producing transitive generators (Claim 6.7 and 6.7A of \cite{Sh430} and section 6 of \cite{AbMag}) rely on the fact that, for every $\lambda\in\pcf(A)\cap N$, we can code a generator $B_\lambda$ for $\lambda$ by way of $N$. More precisely, the key observation is Lemma 5.8 of \cite{AbMag}, which we quote in a simplified form.

\begin{lemma}
Suppose that $A, \chi, N$ are as above with $A\in N_0$. For $\lambda\in \pcf(A)\cap N$, let $\vec{f}=\langle f_\alpha : \alpha<\lambda\rangle$ be a universal cofinal sequence for $\lambda$, and let $\gamma=\sup N\cap \lambda$. Then the set
\begin{equation*}
b_\lambda=\{a\in A : \Ch_N(a)=f_\gamma(a)\}
\end{equation*}
is a generator for $\lambda$.
\end{lemma}  

The generators obtained in this way are subsets of the ordinal closure of $N$, which has cardinality $\kappa$. So suppose we only ask that $\reg(A)<\kappa<\min(A)$, and we somehow manage to obtain transitive generators for some $N\subseteq\pcf_I(A)$ using $\kappa$-presentable structures. In doing so, we will have obtained generators $b_\lambda$ above, which must have size $\kappa$. But then, we can employ \cref{thm: compactness} to see that there are $\lambda_i\in\pcf_I(A)$ for $1\leq i\leq n<\omega$ such that 

\begin{equation*}
A\subseteq_I \bigcup_{1\leq i\leq n}b_{\lambda_i}.
\end{equation*}
\\
So then $I^*$ concentrates on a set of size $\kappa$. As $\pcf_I(A)$ will remain the same if we replace $A$ with an $I$-equivalent set, this amounts to doing pcf theory in the classical case. So in order to have any hope of obtaining transitive generators for more than just the classical case, we would need to use different techniques and perhaps utilize stronger assumptions on $A$, $I$, or even $\pcf_I(A)$ than just $\reg(I)<\min(A)$.

\bibliographystyle{plain}
\bibliography{Trichotomy}
\end{document}